\documentclass[11pt]{amsart}

\newcommand{\cal}[1]{\mathcal{#1}}
\usepackage{john}
\usepackage{hyperref}

\numberwithin{theorem}{section}
\usepackage{color}

\makeatletter
\let\@@pmod\pmod
\DeclareRobustCommand{\pmod}{\@ifstar\@pmods\@@pmod}
\def\@pmods#1{\mkern4mu({\operator@font mod}\mkern 6mu#1)}
\makeatother

\usepackage[T1]{fontenc}
\usepackage[french,english]{babel}

\usepackage[all,cmtip]{xy}
\usepackage[normalem]{ulem}
\usepackage{xcolor}
\newcommand\redsout{\bgroup\markoverwith{\textcolor{red}{\rule[0.5ex]{2pt}{1pt}}}\ULon}
\newcommand\bluesout{\bgroup\markoverwith{\textcolor{blue}{\rule[0.5ex]{2pt}{1pt}}}\ULon}

\date{\today}

\address{John Bergdall\\Department of Mathematics and Statistics \\ Boston University \\ 111 Cummington Mall \\ Boston, MA 02215\\USA}
\email{bergdall@math.bu.edu}
\urladdr{http://math.bu.edu/people/bergdall}

\address{Robert Pollack\\Department of Mathematics and Statistics \\ Boston University \\ 111 Cummington Mall \\ Boston, MA 02215\\USA}
\email{rpollack@math.bu.edu}
\urladdr{http://math.bu.edu/people/rpollack}


\title{A remark on non-integral $p$-adic slopes for modular forms}
\author{John Bergdall and Robert Pollack}

\begin{document}

\begin{abstract}
We give a sufficient condition, namely ``Buzzard irregularity'', for there to exist a cuspidal eigenform which does not have integral $p$-adic slope.\\

\selectlanguage{french}
\noindent \textsc{R\'esum\'e.} \textbf{\textit{Une remarque sur les pentes $p$-adiques non-enti\`eres des formes modulaires.}} On donne une condition suffisante, \`a savoir \og irr\'egularit\'e au sens de Buzzard\fg , pour qu'il existe une forme parabolique propre de pente $p$-adique non-enti\`ere.
\end{abstract}

\selectlanguage{english}

\maketitle

\section{Statement of result}

Let $p$ be a prime number. If $k$ and $M$ are integers then we write $S_k(\Gamma_0(M))$ for the space of weight $k$ cusp forms of level $\Gamma_0(M)$. The $p$-th Hecke operator acting on $S_k(\Gamma_0(M))$ is written $T_p$ if $p\ndvd M$ and $U_p$ otherwise.

For $T = T_p$ or $U_p$, we define the slopes of $T$ to be the slopes of $p$-adic Newton polygon of the inverse characteristic polynomial $\det(1 - TX)$. This is the same as the list of the $p$-adic valuations of the non-zero eigenvalues of $T$, counted with algebraic multiplicity.

To state our theorem we need a definition due to Buzzard \cite{Buzzard-SlopeQuestions}.
\begin{definition}
Let $N \geq 1$ be an integer with $p \ndvd N$.
\begin{enumerate}
\item An odd prime $p$ is $\Gamma_0(N)$-regular if the slopes of $T_p$ acting on $S_k(\Gamma_0(N))$ are all zero for $2 \leq k \leq {p+3\over 2}$.
\item The prime $p=2$ is $\Gamma_0(N)$-regular if the slopes of $T_2$ acting on $S_2(\Gamma_0(N))$ are all zero and the slopes of $T_2$ acting on $S_4(\Gamma_0(N))$ are all either zero or one.
\end{enumerate}
\end{definition}

This definition first appeared in \cite{Buzzard-SlopeQuestions} where Buzzard gives an elementary algorithm, depending on $p$ and $N$, which on input $k$ will output a list of integers.  He conjectures that if $p$ is $\Gamma_0(N)$-regular then this list is exactly the list of slopes of $T_p$ acting on $S_k(\Gamma_0(N))$. The authors of the present work also have made a separate conjecture (\cite{BergdallPollack-GhostPaper}) which  predicts the $U_p$-slopes of all $p$-adic modular forms of tame level $\Gamma_0(N)$ still assuming that $p$ is $\Gamma_0(N)$-regular. The two conjectures are consistent with each other experimentally, but have not yet been shown to be consistent in general.

Buzzard's conjecture clearly implies that every slope is an integer.  (This implication is not at all clear from the conjectures in \cite{BergdallPollack-GhostPaper}.) It is worth asking if the integrality of slopes is characteristic of $\Gamma_0(N)$-regularity. We show that it is. The proof occupies the second section.

\begin{theorem}\label{theorem:main-theorem}
If $p$ is not $\Gamma_0(N)$-regular then there exists an even integer $k$ such that $U_p$ acting on $S_k(\Gamma_0(Np))$ has a slope strictly between zero and one.
\end{theorem}

Coleman theory (which is used below) shows that no harm comes from assuming the witnessing weight in Theorem \ref{theorem:main-theorem} is arbitrarily large. One could try to determine the minimum weight $k$ which confirms Theorem \ref{theorem:main-theorem}. An effective bound should follow from \cite{Wan-GouveaMazur}, but it is likely suboptimal. Numerical data suggest that  the optimal $k$, for $p$ odd, is either $k=j$ or $k = j + (p-1)$ where $2 \leq j \leq {p+3\over 2}$ is a low weight with a non-zero $T_p$-slope.

The theorem is also true for if we replace $U_p$ and $S_k(\Gamma_0(Np))$ by $T_p$ and $S_k(\Gamma_0(N))$. Indeed, if $a_p$ is an eigenvalue for $T_p$ acting on $S_k(\Gamma_0(N))$ then the polynomial $X^2 - a_p X + p^{k-1}$ divides the characteristic polynomial of $U_p$ acting on $S_k(\Gamma_0(Np))$; the eigenvalues $\lambda$ for $U_p$ which are not roots of such polynomials are known to satisfy $\lambda^2 = p^{k-2}$. So, if $k>2$ (which is sufficient by the previous paragraph) the slopes of $U_p$ between zero and one are the same as the slopes of $T_p$ between zero and one.

For $p$ odd, the converse to Theorem \ref{theorem:main-theorem} is also true. Namely, if there exists an even integer $k$ such that $S_k(\Gamma_0(N))$ has a slope strictly between zero and one then $p$ is not $\Gamma_0(N)$-regular. See \cite[Theorem 1.6]{BuzzardGee-SmallSlope}. Its proof uses the $p$-adic local Langlands correspondence for $\GL_2(\Q_p)$ and is thus significantly deeper than the present work. Combining the two results, the following two conditions are equivalent for an odd prime $p$:
\begin{enumerate}
\item The prime $p$ is not $\Gamma_0(N)$-regular.
\item There exists an even integer $k$ such that $T_p$ acting on $S_k(\Gamma_0(N))$ has a slope strictly between zero and one.
\end{enumerate}
There is a natural third condition, implied by (b):
\begin{enumerate}
\setcounter{enumi}{2}
\item There exists an integer $k$ such that $T_p$ acting on $S_k(\Gamma_0(N))$ has a non-integral slope. 
\end{enumerate}
It is conjectured (see \cite{BuzzardGee-Slopes}) that all three conditions are equivalent, but this seems difficult.

\subsection*{Acknowledgements}
The first author was partially supported by NSF grant DMS-1402005 and the second author was partially supported by NSF grant DMS-1303302.

\section{The proof}
We fix algebraic closures $\bar \Q \subset \bar \Q_p$ and write $v_p(-)$ for the induced $p$-adic valuation on $\bar \Q$ normalized so that $v_p(p) = 1$. We also fix an embedding $\bar \Q \subset \C$.  We assume now that $N \geq 1$ is an integer co-prime to $p$.

If $\eta$ is a Dirichlet character of modulus $p$ we write $S_k(\Gamma_1(Np),\eta)$ for the subspace of forms in $S_k(\Gamma_1(Np))$ with character given by $\eta$ ($\eta$ promoted to a character of modulus $Np$). An eigenform $f$ in particular means a normalized eigenform for the standard Hecke operators and the diamond operators. For such an $f$, its $p$-th Hecke eigenvalue is written $a_p(f)$.

Corresponding to the choice of embeddings, each eigenform has an associated two-dimensional $p$-adic Galois representation  $\rho_f: \Gal(\bar \Q/\Q) \goto \GL_2(\bar \Q_p)$. Write $\bar \rho_f$ for its reduction modulo $p$ and $\bar \rho_{f,p}$ (resp.\ $\rho_{f,p}$) for the restriction of $\bar \rho_f$ (resp.\ $\rho_f$) to the decomposition group $\Gal(\bar \Q_p/\Q_p) \subset \Gal(\bar \Q/\Q)$ induced from the embedding $\bar \Q \subset \bar \Q_p$.  Note that the construction of $\bar\rho_f$ requires the choice of a Galois-stable lattice, but that the semi-simplification of $\bar\rho_f$ is independent of this choice. In particular, whether or not $\bar\rho_{f,p}$ is irreducible is also independent of the choice of a stable lattice.

\begin{lemma}\label{lemma:integral-implies-reducible}
Let $\eta$ be a Dirichlet character of conductor $p$ and $f$ an eigenform in $S_2(\Gamma_1(Np),\eta)$. If $v_p(a_p(f))$ equals $0$ or $1$,  then $\rho_{f,p}$ is reducible.
\end{lemma}
\begin{proof}
If $v_p(a_p(f)) = 0$ then it is well known that $\rho_{f,p}$ is reducible. For example, see \cite[Lemma 2.1.5]{Wiles-OrdinaryModular} and the references therein. (This is also commonly attributed to a letter from Deligne to Serre in the 1970s which has never been made public.)

Now suppose that $v_p(a_p(f)) = 1$. Then, there is an eigenform $f'$ in $S_2(\Gamma_1(Np),\eta^{-1})$ with   $v_p(a_p(f')) = 0$ and $\rho_f$ isomorphic to $\rho_{f'}$ up to a twist. (The form $f'$ is sometimes called the Atkin--Lehner involute of $f$; see \cite[Proposition 3.8]{BergdallPollack-FredholmSlopes}.) Since the first argument applies to $f'$, we deduce that $\rho_{f',p}$ and its twist $\rho_{f,p}$ are both reducible.
\end{proof}

\begin{proposition}\label{proposition:fractional-slopes-wt-2}
If $p$ is odd and not $\Gamma_0(N)$-regular then there exists an even Dirichlet character $\eta$ of modulus $p$ such that $U_p$ acting on $S_2(\Gamma_1(Np),\eta)$ has a slope strictly between zero and one.
\end{proposition}
\begin{proof}
Choose an integer $2 \leq k \leq {p+3\over 2}$ and an eigenform $f \in S_k(\Gamma_0(N))$ with $v_p(a_p(f)) > 0$. By \cite[Theorem 2.6]{Edixhoven-Weights}, $\bar\rho_{f,p}$ is irreducible.

Suppose first that $f$ has weight 2. Then, the polynomial $X^2 - a_p(f)X + p$ divides the characteristic polynomial of $U_p$ acting on $S_k(\Gamma_0(Np))$ (as in the remarks after Theorem \ref{theorem:main-theorem}). The theory of the Newton polygon implies that the roots of this polynomial have valuation strictly between zero and one, so we can choose $\eta$ to be the trivial character and we are done in this case.

Now assume that $f$ has weight at least $4$ and thus also $p\geq 5$. By \cite[Theorem 3.5(a)]{AshStevens-Duke}, which assumes $p\geq 5$, there exists an even Dirichlet character $\eta$ necessarily of conductor $p$ (because $f$ has weight at most ${p+3\over 2} < p+1$) and an eigenform $g \in S_2(\Gamma_1(Np),\eta)$ such that $\bar \rho_g$ and $\bar \rho_f$ have isomorphic semi-simplifications. Since $\bar\rho_{f,p}$ is irreducible, $\bar\rho_{g,p}$ is as well.  Thus, $\rho_{g,p}$ is irreducible, and Lemma \ref{lemma:integral-implies-reducible} implies that $v_p(a_p(g))$ is strictly between zero and one.
\end{proof}
Proposition \ref{proposition:fractional-slopes-wt-2} is an analog of Theorem \ref{theorem:main-theorem} for weight two forms with character, and its proof confirms our theorem  when there is a weight 2 form of level $\Gamma_0(N)$ with positive $T_p$-slope. To prove Theorem \ref{theorem:main-theorem} in general, we use  the theory of $p$-adic modular forms. We refer to \cite{Coleman-pAdicBanachSpaces} for the facts in the next two paragraphs.

If $\kappa: \Z_p^\times \goto \bar \Q_p^\times$ is a continuous character (a ``$p$-adic weight'') then we write $S_{\kappa}^{\dagger}(N)$ for the space of overconvergent $p$-adic cusp forms of weight $\kappa$ and tame level $\Gamma_0(N)$ equipped with its $U_p$-operator. If $k$ is an integer and $\kappa(z) = z^k$ then we write this space as $S_k^{\dagger}(N)$; it contains $S_k(\Gamma_0(Np))$ as a $U_p$-compatible subspace. Likewise, if $\kappa(z) = z^k\eta(z)$ where $\eta$ is a non-trivial finite order character of $\Z_p^\times$ then $S_{z^k\eta}^{\dagger}(N)$ contains $S_k(\Gamma_1(Np^{f_\eta}),\eta)$ as a $U_p$-compatible subspace (where $p^{f_\eta}$ is the conductor of $\eta$). 

By Coleman theory we mean the following:\ suppose that $\kappa$ is a $p$-adic weight and $h$ is the $p$-adic valuation of a non-zero eigenvalue for $U_p$ appearing in $S_{\kappa}^{\dagger}(N)$. Then, for any sequence of $p$-adic weights $(\kappa_n)_{n\geq 0}$ such that $\kappa_n$ and $\kappa$ agree on the torsion subgroup of $\Z_p^\times$, and $\kappa_n(1+2p) \goto \kappa(1+2p)$ as $n \goto \infty$, we have that $h$ is also a $U_p$-slope in $S_{\kappa_n}^{\dagger}(N)$ for $n\gg 0$.

We can now give the proof of the theorem.

\begin{proof}[Proof of Theorem \ref{theorem:main-theorem}]
Assume first that $p$ is odd. By Proposition \ref{proposition:fractional-slopes-wt-2} there exists an even Dirichlet character $\eta$ of modulus $p$ and rational number $0 < h < 1$ which appears as a $U_p$-slope in $S_2(\Gamma_1(Np),\eta)$. Thus, the slope $h$ appears as a $U_p$-slope in $S_{z^2\eta}^{\dagger}(N)$.  Choose $j\geq 0$ even so that $\eta|_{\F_p^\times}$ is of the form $z \mapsto z^j$. Then, for $n \gg 0$ and $k_n = 2 + j + (p-1)p^n$, the slope $h$ is a $U_p$-slope in $S_{k_n}^{\dagger}(N)$ by Coleman theory described above. For such $k$ we have $h < 1 < k-1$ and so $h$ is $U_p$-slope in $S_k(\Gamma_0(Np))$ by \cite[Theorem 6.1]{Coleman-ClassicalandOverconvergent}.

The proof for $p=2$ is similar to the argument in Proposition \ref{proposition:fractional-slopes-wt-2} when $k=2$. If either $S_2(\Gamma_0(N))$ or $S_4(\Gamma_0(N))$ has a non-integral slope we are done. If not,  then either $S_2(\Gamma_0(N))$ contains a slope one form, or $S_4(\Gamma_0(N))$ contains a form of slope two or three. In either case, the corresponding $2$-adic refinements will have fractional slope.
\end{proof}

\bibliography{ghost_bib}
\bibliographystyle{abbrv}

\end{document}